\documentclass[12pt]{amsart}
\usepackage[paperwidth=210mm,paperheight=297mm,inner=2.8cm,outer=2.8cm,top=2.8cm,bottom=2.8cm]{geometry}

\makeatletter
\def\thm@space@setup{\thm@preskip=\baselineskip \thm@postskip=\thm@preskip}
\def\section{\@startsection{section}{1}\z@{1.25\baselineskip}{1\baselineskip}{\normalfont\scshape\centering}}
\makeatother
\usepackage{amssymb}

\renewcommand\P{\mathbb P}
\newcommand{\PP}{\mathbb P}

\newcommand\CC{{\mathbb C}}

\newcommand\KK{{\mathbb K}}

\newcommand\be{\begin{eqnarray*}}
\newcommand\ee{\end{eqnarray*}}
\newenvironment{state}[1]{\begin{trivlist}\item[]\textbf{#1.}\itshape\enspace\ignorespaces}{\end{trivlist}}

\begin{document}

\title{Counting lines on projective surfaces}

\author{Thomas Bauer}
\address{
Fachbereich Mathematik und Informatik,
Philipps-Universit\"at Marburg,
Hans-Meerwein-Stra\ss e,
35032 Marburg, Germany.
}
\email{tbauer@mathematik.uni-marburg.de}

\author{S{\l}awomir Rams}
\address{
Institute of Mathematics,
Jagiellonian University,
ul. {\L}ojasiewicza 6,
30-348 Krak\'ow,
Poland
}
\email{slawomir.rams@uj.edu.pl}
\thanks{Research partially supported by the National Science Centre, Poland, Opus grant no.~2017/25/B/ST1/00853 (S. Rams)}


\begin{abstract}
In this note we prove a new bound on the number of lines on a smooth surface of degree $d\ge 3$ in
$\PP^{3}$.
Building on work of Segre, we provide a rigorous justification of an idea of his while at the same time
improving his bound.
Our result gives the lowest known bound for $d \geq 6$, and it is
valid both in characteristic 0 and in positive characteristic $p>d$.
\end{abstract}

\maketitle

\newcommand{\XXd}{X_{d}}
\newcommand{\XXp}{X_{5}}
\newcommand{\Pl}{\Pi}
\newcommand{\reg}{\operatorname{reg}}
\theoremstyle{definition}
\newtheorem{definition}{Definition}[section]
\theoremstyle{plain}
\newtheorem{obs}{Observation}[section]
\theoremstyle{definition}
\newtheorem{rem}[obs]{Remark}
\newtheorem{examp}[obs]{Example}
\theoremstyle{definition}
\newtheorem{defi}[obs]{Definition}
\theoremstyle{plain}
\newtheorem{prop}[obs]{Proposition}
\newtheorem{theo}[obs]{Theorem}
\newtheorem{lemm}[obs]{Lemma}
\newtheorem{mainlemma}[obs]{Main Lemma}
\newtheorem*{conj}{Conjecture}
\newtheorem{claim}[obs]{Claim}
\newtheorem{cor}[obs]{Corollary}
\newtheorem*{conv}{Convention}
\newcommand{\ux}{\underline{x}}
\newcommand{\ud}{\underline{d}}
\newcommand{\ue}{\underline{e}}
\newcommand{\mmS}{{\mathcal S}}
\newcommand{\mmP}{{\mathcal P}}
\newcommand{\mathbbT}{T}
\newcommand{\nlines}{\ell(\XXp)}
\newcommand{\ii}{\operatorname{i}}

\newcommand{\nonlinflec}{{\mathcal D}}
\newcommand{\linflec}{{\mathcal L}}
\newcommand{\flec}{{\mathcal F}}
\section{Introduction}

In this note we establish a new bound on the number of lines that
can lie on a smooth projective surface in $\P^3$ of given degree.

Recently there has been quite some interest in configurations of lines on surfaces in $\PP^{3}(\CC)$
(see e.g. \cite{dis-2015},    \cite{kollar},  \cite{shioda},  \cite{sarti},  \cite{miyaoka}, \cite{benedetti-regularity}).
In particular, the picture of the geometry of line configurations on complex projective quartic surfaces is complete (up to \cite[Conjecture~4.7]{gonzalezalonso}).
The claim that  the maximal number of lines on a smooth quartic is $64$
can be found in  \cite{segre43}, whereas the first correct proof of that fact is given in \cite{RS}.
The paper \cite{dis-2015} contains
a complete classification of smooth complex quartic surfaces with
many lines. Finally, lines on complex quartics with singular points  are  considered in \cite{veniani-phd},  \cite{gonzalezalonso}.

By contrast,  the maximal number of lines on smooth hypersurfaces in $\PP^{3}(\CC)$ of a fixed degree $d \geq 5$  remains unknown
(see \cite{segre43}, \cite{sarti}, \cite{miyaoka}, \cite{eisenbud-harris-3264}). In the case of smooth quintic surfaces the proof of the inequality
$$
\ell(X_5) \leq 127
$$
can be found in the  recent paper \cite{ramsschuett}, whereas (until now) the best bound for smooth complex surfaces of  degree $d \geq 6$ has been the inequality
\begin{equation} \label{eq-segrebound}
\ell(X_d) \leq (d-2)(11d-6)
\end{equation}
that was stated by Segre in  \cite[$\mathsection$ 4]{segre43}.

Segre's proof of \eqref{eq-segrebound} in \cite{segre43} is based on various properties of so-called lines of the second kind.
In particular, he states that every line of the second kind comes up in the flecnodal divisor with multiplicity two (see the next section for details).
Unfortunately, several claims made in \cite{segre43}
are false (\cite[$\mathsection$~3]{rs-advgeom}) and the proof of the claim on the multiplicity of lines of the second kind does not fulfill modern standards of rigor
(see Remark~\ref{rem-segre-wrong}.b).

Our original aim while working on this note was to determine whether Segre's brilliant idea can be proven using modern algebraic geometry (and in particular, whether it is correct).
We were able to show that this is the case, and in fact we improved on his bound.
   Our result addresses the case of characteristic zero as well
   that of positive characteristic.
   Under suitable assumptions on the degree,
   we can use
   a result of Voloch \cite{voloch} to conclude
   that a general point is not flecnodal, even in positive characteristic.
   We show:

\begin{theo} \label{thm-bound-char-p}
   Let $\XXd \subset \PP^{3}(\KK)$ be a smooth surface of degree $d \geq 3$
   over a field of
   characteristic $0$ or of characteristic $p > d$.
   Let $\ell(\XXd)$ be the number of lines that the surface  $\XXd$ contains.
   Then the following inequality holds
   \begin{equation} \label{eq-bound-char-p}
      \ell(\XXd) \leq 11 d^2 - 30d + 18
       \, .
   \end{equation}
\end{theo}

This result provides the lowest known bound on the number of lines lying on a degree-$d$ surface for $d \geq 6$.
Still, the question what is the maximal number of lines on smooth projective surfaces of a fixed degree $d \geq 5$ remains open.
Note that the Fermat quartic in characteristic $3$ contains $112$ lines -- this
shows that the assumption $p>d$ cannot be omitted.
(see e.g.~\cite{RS112}).

The first bound on the number of lines on a smooth degree-$d$ surface was stated by Clebsch:
\begin{equation} \label{eq-clebsch-salmon}
\ell(X_d) \leq d(11d-24)
\end{equation}
(\cite[p.~106]{Clebsch}), who used ideas coming from Salmon (\cite[p.~95]{Clebsch}, \cite{salmon}). While a beautiful modern proof of \eqref{eq-clebsch-salmon} is given in \cite[$\mathsection$ 11.2.1]{eisenbud-harris-3264}, we follow a different path to arrive at
\eqref{eq-bound-char-p} (see Remark~\ref{rem-segre-wrong}.b), so the only fact from \cite{eisenbud-harris-3264} we use is the result,
due to McCrory and Shifrin \cite{mcrory-shifrin},
that a general point of the surface $X_d$ is not flecnodal (Lemma~\ref{lemm-two}).
In characteristic zero, a bound on the number of lines on
complete intersections of codimension two or more can be derived from the orbibundle Miyaoka-Yau-Sakai inequality, but this approach yields no results on the codimension-one case  (see \cite[Remark~3 on p.~921]{miyaoka}).

   We do not believe that the particular bound in
   Theorem~\ref{thm-bound-char-p} is sharp. Note, in particular, that it
   would be
   in
   line with the results on quartic surfaces
   \cite{RS,RS112}
   that surfaces might generally be able to
   carry more lines when the characteristic is positive.

\begin{conv} \label{conv}
   In this note we work over an algebraically closed  field $\KK$ of characteristic $p$,
   where either $p=0$ or $p>d$, where $d \ge 3$ is a fixed integer.
\end{conv}

%*****************************************************************************

\section{Lines as multiple components of the flecnodal curve}

Let $\XXd \subset \PP^{3}(\KK)$ be a smooth degree-$d$ surface, where $d \geq 3$,
and let $f$ be a generator of its ideal ${\mathcal I}(X_d)$.

For a line $L \subset \PP^{3}$ we put  $\ii(P,L.\XXd)$ to denote the order of vanishing of the restriction
  $f|_L$ at the point $P$ if $L \nsubseteq \XXd$.
We define $\ii(P,L.\XXd) := \infty$ when $P \in L \subset \XXd$.

Recall that a line $L \subset X_d$ is called  {\em a line of the second kind} if
it meets every plane curve $\Gamma \in  |\mathcal{O}_{\XXd}(1)- L|$ only in  inflection points of the latter
(see \cite[p.~87]{segre43}, \cite{rs-advgeom}).
Otherwise, the line  $L$ is called  {\em a line of the first kind}.

In the proof of Theorem \ref{thm-bound-char-p}
the following proposition -- and in particular
its statement about the multiplicities of the lines of the second kind --
plays a crucial role (see \cite[p.~90]{segre43} for Segre's claim in this direction).

\begin{prop} \label{thm-flecnodal}
Let $\XXd \subset\P^3(\KK)$ be a smooth surface of degree $d>2$, where $\KK$ is
an algebraically closed field of characteristic $p$. Let $p=0$ or
$p>d$.
Then there exists an (effective) Weil divisor  $\flec(\XXd) \in |{\mathcal O}_{\XXd}(11d-24)|$ such that  the equality
$$\operatorname{supp}(\flec(\XXd)) = \{ P \in \XXd \, :  \mbox{there exists a line } L \mbox{ such that } \ii(P,L.\XXd) \geq 4 \}  $$
holds, and each line  $L \subset \XXd$ of the second kind appears in $\flec(\XXd)$ with multiplicity at least two.
\end{prop}
In the sequel, we call $\flec(\XXd)$ the {\em flecnodal divisor of the surface } $\XXd$.

The proof of Prop.~\ref{thm-flecnodal} will be preceded by several
lemmata. First, we introduce the necessary notation.
For  $j=1,2,3$ we define % bihomogeneous
 polynomials ${\mathfrak t}^{(j)} \in \KK[w_0, \ldots, w_3, z_0, \ldots, z_3]$ by the formula
\begin{equation} \label{eq-def-tj}
{\mathfrak t}^{(j)} := \sum_{0 \leq i_1,\ldots,i_j \leq 3}   \frac{\partial f^{j}}{\partial w_{i_1} \ldots \partial w_{i_j}}(w_0, \ldots, w_3) \cdot z_{i_1} \cdot \ldots \cdot z_{i_j} \, .
\end{equation}
 In order to simplify our notation, given a point $P = (p_0, \ldots, p_3) \in \KK^4$ (resp. $P = (p_0: \ldots: p_3) \in \PP^3$) we put
$$
{\mathfrak t}_P^{(j)}(z_0, \ldots, z_3) := {\mathfrak t}^{(j)}(p_0, \ldots, p_3, z_0, \ldots, z_3).
$$
Observe that the zero set of the polynomial ${\mathfrak t}_P^{(1)}$ (resp. ${\mathfrak t}_P^{(2)}$) is
 the projective tangent space  ${\mathbbT}_P \XXd$ (resp. the Hessian quadric  $\mbox{V}_P =\mbox{V}_P \XXd$).

We consider the variety
$$\mmP  := \mmP(\XXd) \subset \XXd \times \mbox{G}(2,4)$$
defined as
$$
\mmP := \{ \, (P,L) \, :  \, P \in \XXd \mbox{ and  the line } L \mbox{ satisfies the condition } \ii(P,L.\XXd) \geq 3 \} \, ,
$$
where $\mbox{G}(2,4)$ is the Grassmanian of lines in  $\PP^{3}(\KK)$. The variety $\mmP$ is endowed with the projections
$$
\pi_1 \, : \, \mmP \rightarrow \XXd  \quad \mbox{and} \quad \pi_2 \, : \, \mmP \rightarrow  \mbox{G}(2,4) \, .
$$

\begin{lemm} \label{lemma-one}

\noindent
{\rm (a)} If $P \in \XXd$, then $\# \pi_1^{-1}(P) \in \{1, 2, \infty \}$.

\noindent
{\rm (b)} The set $\{ P \in \XXd \, : \, \# \pi_1^{-1}(P) = \infty \}$ is finite.
\end{lemm}
\begin{proof}  (a) We fix a point $P \in \XXd$.
Obviously, either ${\mathfrak t}_P^{(2)}|_{{\mathbbT}_P \XXd} \equiv 0$ or
${\mathfrak t}_P^{(2)}|_{{\mathbbT}_P \XXd}$ vanishes along exactly one or two lines.

Let $L \subset  {\mathbbT}_P \XXd$ be a line.
By direct computation, ${\mathfrak t}_P^{(2)}$ vanishes along $L$ if and only if
$\ii(P,L.\XXd) \geq 3$.
Thus a finite  fiber $\pi_1^{-1}(P)$ consists of at most two points,   which yields (a).

\noindent
(b) Suppose that there exists a curve $C \subset \XXd$ such that
$$
{\mathfrak t}_P^{(2)}|_{{\mathbbT}_P \XXd} \equiv 0 \mbox{ for every point } P \in C.
$$
Then the Gauss map
is constant on the curve $C$, because its differential vanishes for all $P \in \mbox{reg}(C)$. Thus
the curve $C$ is contained in  a fiber of the Gauss map, % $\gamma$,
 which is impossible by Zak's result \cite[Thm.~2.3]{zak}.
\end{proof}

In particular, we have shown that the variety $\mmP$ is two-dimensional.

We will need to know that not all points on a smooth surface can
be flecnodal. In characteristic $p=0$ this result is due
to McCrory and Shifrin (see \cite[Lemma~2.10]{mcrory-shifrin}, \cite[Prop.~11.8]{eisenbud-harris-3264}), whereas
for characteristic $p > d$ the lemma below was shown by Voloch (see \cite[Theorem 1 and Prop.~1]{voloch}):

\begin{lemm} \label{lemm-two} % {\rm (\cite[Lemma~2.10]{mcrory-shifrin}, \cite[Prop.~11.8]{eisenbud-harris-3264})}
Let  $\XXd \subset\P^3(\KK)$ be a smooth surface of degree $d>2$, where $\KK$ is
an algebraically closed field of characteristic $p$, and let $p=0$ or
$p>d$. Then
$$
\XXd \neq \{ P \in \XXd \, : \mbox{there exists a line } L \mbox{ such that } \ii(P,L.\XXd) \geq 4 \}.
$$
\end{lemm}

In the sequel, the divisor defined by   ${\mathfrak t}^{(j)}$  in $\PP^{3} \times \PP^{3}$ is denoted by ${\mathfrak T}^{(j)}$ and we put
$$
Y_d := (\XXd \times \PP^{3}) \cap \bigcap_{j=1}^{j=3} {\operatorname{supp}}({\mathfrak T}^{(j)}) \, .
$$
Moreover, $\psi \, : \, \PP^{3} \times \PP^{3} \rightarrow \PP^{3}$
stands for the projection onto the first factor and  we define
$$
\Delta_{\XXd} := \{ (P,P) \in \PP^{3} \times \PP^{3} \, : \, P \in \XXd \}.
$$
\begin{lemm} \label{lemm-three}
The variety $Y_d$ is two-dimensional.
\end{lemm}
\begin{proof} By Lemma~\ref{lemma-one}.b, only finitely many fibers of the projection $\psi|_{Y_d}$
are two-dimensional.

Fix a point $P \in \XXd$ such that  the fiber $\pi_1^{-1}(P)$ is finite, i.e., ${\mathfrak t}_P^{(2)}$ vanishes
on  ${\mathbbT}_P \XXd$ only along two
 (not necessarily distinct) lines $L_1, L_2$. By direct computation
\begin{equation} \label{eq-useful}
L_1 \subset  \operatorname{V}({\mathfrak t}_P^{(3)}) \quad \mbox{ iff } \quad \ii(P,L_1.\XXd) \geq 4 \, .
\end{equation}
Thus, by Lemma~\ref{lemm-two}, for generic choice of $P \in \XXd$ the fiber of $(\psi|_{Y_d})^{-1}(P)$
is finite.
\end{proof}
Lemma~\ref{lemm-three} implies that the proper intersection $2$-cycle
\begin{equation} \label{eq-intersectioncycle}
{\mathfrak T}^{(1)} \cdot {\mathfrak T}^{(2)} \cdot {\mathfrak T}^{(3)} \cdot (\XXd \times \PP^{3})
\end{equation}
is well-defined (see e.g. \cite[Chap.~V.C.2]{serre}). One of its components is the variety  $\Delta_{\XXd}$.
\begin{lemm} \label{lemm-four}
The diagonal $\Delta_{\XXd}$ comes up in the intersection cycle \eqref{eq-intersectioncycle}
with multiplicity~$6$.
\end{lemm}
\begin{proof}
We are
to show
that, for generic choice  of the point $P \in \XXd$,  the intersection multiplicity of  the curves $V({\mathfrak t}_P^{(j)}|_{{\mathbbT}_P \XXd})$,
where $j = 2,3$, in  $P$ equals $6$. By Lemma~\ref{lemma-one}.(b) we can assume that $V({\mathfrak t}_P^{(2)}|_{{\mathbbT}_P \XXd})$ consists of two lines $L_1, L_2$.
Moreover, Lemma~\ref{lemm-two} and \eqref{eq-useful} allow us to require that
$$
{\mathfrak t}_P^{(3)}|_{L_k}  \mbox{ does not vanish identically }  \mbox{ for } k= 1,2.
$$
Then, by direct computation,   the restriction
 ${\mathfrak t}_P^{(3)}|_{L_k}$ has a triple root in $P$ for $k =1,2$ and the proof is complete.
\end{proof}
Given a principal line $L_k \subset T_P\XXd$,  the above proof shows that
\begin{equation} \label{eq-two-possibilities}
\mbox{ either } {\mathfrak t}_P^{(3)}|_{L_k} \equiv 0 \quad \mbox{ or } P \mbox{ is the unique zero of } {\mathfrak t}_P^{(3)}|_{L_k}
\end{equation}

In particular, Lemma~\ref{lemm-four}, all components appear in the cycle
\begin{equation} \label{eq-def-w}
{\mathfrak W} :=  {\mathfrak T}^{(1)} \cdot {\mathfrak T}^{(2)} \cdot {\mathfrak T}^{(3)} \cdot (\XXd \times \PP^{3}) - 6 \Delta_{\XXd}
\end{equation}
with non-negative coefficients. Moreover, by definition,
the set $\psi(\operatorname{supp}({\mathfrak W}))$ consists of the points $P$ such that ${\mathfrak t}_P^{(1)}$, $\ldots$, ${\mathfrak t}_P^{(3)}$
vanish simultaneously along a line. From \eqref{eq-useful} and
 \eqref{eq-two-possibilities},
 we obtain the equality
\begin{equation} \label{eq-pushforwardisok}
 \psi({\operatorname{supp}({\mathfrak W})}) = \{ P \in \XXd \, : \mbox{ there exists a line } L \mbox{ such that } \ii(P,L.\XXd) \geq 4 \}.
\end{equation}

Let $L \subset  \XXd$ be a line.
Recall that the linear
system $|{\mathcal O}_{\XXd}(1) - L|$ endows the surface in question with a  fibration $$\pi \, : \,  \XXd \to \PP^1.$$
Let us follow \cite{segre43} and put $\Gamma_P$ to denote its fiber that is contained in the tangent space
${\mathbbT}_P \XXd$
for a point $P \in L$. One can easily check that if the Hessian quadric $\mbox{V}_P$
does not contain the tangent space in question, then the line residual to $L$ in the (scheme-theoretic) intersection ${\mathbbT}_P \XXd \cap V_P$
is tangent to the curve  $\Gamma_P$ in the point $P$.

For the proof of Prop.~\ref{thm-flecnodal} we will need the following observation.
\begin{lemm} \label{lemm-non-degenerate}
 Let $L \subset  \XXd$ be a line. Then the set
$$
\{ P \in L \, : \, V({\mathfrak t}_P^{(1)})  \cap V({\mathfrak t}_P^{(2)}) \mbox{ does not consist of two distinct lines } \}
$$
is finite.
\end{lemm}
\begin{proof} %(c.f. \cite[Lemmata~3.1,~3.2]{ramsschuett-quintics})
Observe that our assumption on the base field $\KK$ combined with \cite[Prop.~IV.2.5]{hartshorne} implies that the map $\pi|_L \, : \, L \rightarrow \PP^1$ (i.e., the restriction of the fibration $\pi$
to the line $L$) is separable.

Let $P \in L$ be a point. By Lemma~\ref{lemma-one}, we can assume that the Hessian quadric $V_P$ does not contain the tangent space
${\mathbbT}_P \XXd$. If the Hessian quadric and the tangent space meet along the line $L$ with multiplicity two, then $L$ is tangent to the curve  $\Gamma_P$
in the point $P$. Thus $P$ is the ramification point of the degree-$(d-1)$ map $\pi|_L$ % \, : \, L \rightarrow \PP^1$ (i.e., the restriction of the fibration $\pi$ to the line $L$).
(one can repeat verbatim the proof of \cite[Lemmata~3.1,~3.2]{ramsschuett}).  But  $\pi|_L$ is generically etale, so it has  only finitely many ramification points, and the claim follows.
\end{proof}

One has the following property  of lines of the second kind.

\begin{lemm} \label{lemm-second-kind}
Let  $L \subset  \XXd$ be a line and let  $P \in L$ be a point such that the Hessian quadric
 $\mbox{V}_P$
does not contain the tangent space ${\mathbbT}_P \XXd$.
If the line $L$ is of the second kind, then
the form ${\mathfrak t}_P^{(3)}$ vanishes along the (set-theoretic) intersection  ${\mathbbT}_P \XXd \cap V_P$.
\end{lemm}
\begin{proof} % (c.f. \cite[Prop.~3.9]{ramsschuett-quintics})
We can assume that the plane and the quadric meet along two distinct lines,
one of which is not contained in $\XXd$ (otherwise the claim is obvious, because ${\mathfrak t}_P^{(3)}$ vanishes along  $L \subset  \XXd$).

Let $L'$ be the line residual to $L$ in ${\mathbbT}_P \XXd \cap V_P$. As we already explained, the line $L'$ is tangent to the fiber  $\Gamma_P$ of $\pi$ in the point $P$.
Since $L$ is of the second kind, $P$ is an inflection point of the curve $\Gamma_P$, so  $L'$ meets  $\Gamma_P$ with multiplicity at least $3$ in the point $P$.
But $L \subset \XXd$ also meets $L'$ in the point $P$, so we have
$$
\ii(P,L'.\XXd) \geq 4
$$
The claim follows directly from \eqref{eq-useful}.
\end{proof}

Now we are in position to give a proof of Prop.~\ref{thm-flecnodal}. In the proof below we maintain the notation of this section.
In particular the cycle ${\mathfrak W}$ is given by   \eqref{eq-def-w}, and $\psi$ denotes the projection $\PP^{3} \times \PP^{3} \rightarrow \PP^{3}$ onto the first factor.

\begin{proof}[Proof of Prop.~\ref{thm-flecnodal}]
Let $H \subset \PP^{3}$ be a generic hyperplane. We claim that the effective divisor
$$
\flec(\XXd) := \psi_{*}({\mathfrak W} \cdot (\PP^{3} \times H))
$$
has the required properties.

Indeed, one can easily see that
$$
\psi({\operatorname{supp}({\mathfrak W})}) = \psi({\operatorname{supp}({\mathfrak W})} \cap (\PP^3 \times H)),
$$
so \eqref{eq-pushforwardisok} implies that the support of  $\flec(\XXd)$ satisfies the claim of the proposition.

In order to show that  $\flec(\XXd) \in |{\mathcal O}_{\XXd}(11d-24)|$ we  compute the class
$$
[\psi_{*}({\mathfrak W} \cdot (\PP^{3} \times H))]
$$
in the Chow ring $A^{*}(\PP^{3})$.
We put $h_1 := [H \times \PP^{3}]$ and $h_2 := [\PP^{3} \times H]$.
At first we compute in $\mbox{A}^{*}(\PP^{3} \times \PP^{3})$:
\begin{eqnarray*}
[{\mathfrak T}^{(1)}] \cdot [{\mathfrak T}^{(2)}] \cdot [{\mathfrak T}^{(3)}] &=&  ((d-1) \, h_1 + h_2) \cdot ((d-2) \, h_1 + 2 \, h_2) \cdot ((d-3) \, h_1 + 3 \, h_2) \\
 &  = & 6 \, h_2^3 + (11d-18) \, h_2^2 \cdot h_1 \\
 &    & {} + (\mbox{terms of degree at most one w.r.t. } h_2)
\end{eqnarray*}
Since $[\Delta_{\PP^{3}}].\psi^{*}([\XXd]) = [\Delta_{\XXd}]$ we have
\begin{equation} \label{eq-firststep}
\psi_{*}([{\mathfrak W}] \cdot h_2)  = \psi_{*}( ([{\mathfrak T}^{(1)}] \cdot [{\mathfrak T}^{(2)}] \cdot [{\mathfrak T}^{(3)}]- 6 [\Delta_{\PP^{3}}]) \cdot \psi^{*}[\XXd] \cdot h_2).
\end{equation}
Recall that, by  \cite[Ex.~8.4.2]{FultonInters}, the class $[\Delta_{\PP^{3}}]$ of the diagonal in $\PP^{3} \times \PP^{3}$
can be expressed as
\begin{equation*} \label{eq-deltapn}
[\Delta_{\PP^{3}}] = h_1^3 + h_1^2\cdot h_2 + h_1.h_2^2 + h_2^3   \, ,
\end{equation*}
whereas $h_1 = \psi^{*}[H]$. Therefore, from \eqref{eq-firststep}, we obtain that
\begin{align*}
\psi_{*}([{\mathfrak W}] \cdot h_2) &= (\psi_{*}([{\mathfrak T}^{(1)}] \cdot  [{\mathfrak T}^{(2)}]
 \cdot [{\mathfrak T}^{(3)}] ) \cdot h_2) - 6 \, \psi_{*}([\Delta_{\PP^{3}}] \cdot h_2) \cdot [\XXd] \\
 &= (11d-24) \, {\mathcal O}_{\PP^{3}}(1) \cdot  [\XXd]
\end{align*}
and the proof of that part of the proposition   is complete.

Finally, let $L \subset  \XXd$ be a line of the second kind and let $P \in L$ be a point such that
 ${\mathbbT}_P \XXd$ and the Hessian quadric $V_P$ meet along two distinct lines.
As in Lemma~\ref{lemm-second-kind} we put $L'$ to denote the line residual to $L$ in ${\mathbbT}_P \XXd \cap V_P$.
We can assume that the hyperplane $H$ meets  $L$ (resp. $L'$) in the point $Q \neq P$ (resp. $Q' \neq P$).
Obviously we have $(P,Q) \in {\operatorname{supp}({\mathfrak W})} \cap (\PP^3 \times H)$.
Moreover, by Lemma~\ref{lemm-second-kind}, the point  $(P,Q')$ also belongs to the
set ${\operatorname{supp}({\mathfrak W})} \cap (\PP^3 \times H)$. Since $Q \neq Q'$, Lemma~\ref{lemm-non-degenerate} implies
that the restriction of the projection $\psi$,
$$
\psi^{-1}(L) \cap {\operatorname{supp}({\mathfrak W})} \cap (\PP^3 \times H) \rightarrow L \,,
$$
is of degree at least two, so the claim on the multiplicity follows from the definition of the map $\psi_{*}$.
\end{proof}

\begin{examp} \label{example-schur} %{\bf Maybe to be written in a precise way}
An elementary computation shows that the Schur quartic
$$
x_0^4-x_0 x_1^3=x_3^4-x_3 x_4^3
$$
contains exactly $64$ lines:
$48$ lines of the first kind and $16$ lines of the second kind. Since the flecnodal
divisor of a quartic surface has degree $80$,
 each line of the second kind
must come in the flecnodal divisor precisely with multiplicity two. Thus, the lower bound of Prop.~\ref{thm-flecnodal} is sharp.
\end{examp}

\begin{rem} \label{rem-segre-wrong}
(a) The idea of studying lines on a surface via points of fourfold contact goes back to work of Salmon and Clebsch on cubic surfaces (see \cite{kollar} and the bibliography therein). In particular,
an equation of the flecnodal divisor is obtained in \cite{Clebsch} via  projection of the intersection of the varieties  ${\mathfrak T}^{(j)}$.
A beautiful exposition of a modern treatment of this approach  can be found in \cite[$\mathsection$ 11.2.1]{eisenbud-harris-3264}. Still, for the proof of
 Prop.~\ref{thm-flecnodal}, we find it more convenient to avoid the use of bundles of relative principal parts. In this way we can
control the  behaviour of the flecnodal divisor along a line of the second kind.

(b) As we already explained, the claim on multiplicities of lines of the second kind in the flecnodal divisor
was stated in \cite[p.~90]{segre43}. Segre (see \cite[(7) on p.~88]{segre43}) justified it by giving an explicit formula for an analytic function
(defined on an open neighbourhood of a point $P$ on a line $L \subset \XXd$)
that vanishes along the set $\mbox{supp}(\flec(\XXd))$  and showing that the function in question has multiplicity at least two along the line $L$ provided the latter is of the second kind.
Unfortunately, this % brilliant
argument does not explain why the function in question is a local equation of the flecnodal divisor (although it explains why its set of zeroes contains the  support $\mbox{supp}(\flec(\XXd))$), i.e., it does not explain why its order
of vanishing along the line $L$ yields
any information on the multiplicity with which $L$ comes up in the divisor $\flec(\XXd)$.
\end{rem}

%*****************************************************************************

\section{Bound on the number of lines}

We recall the following fact that we need for the proof of Thm.~\ref{thm-bound-char-p}.

\begin{claim} \label{claim-first-kind}  {\rm (\cite[p.~88]{segre43})} Assume that the characteristic $p$ of the ground field is either zero or bigger than $d$. A line $L \subset \XXd$ of the first kind is met by at most $(8d-14)$
other lines lying on the surface $X_d$.
\end{claim}

Given a (Weil) divisor $Z=\sum_j \alpha_j C_j$ on $\XXd$ and a plane $\Pi \subset \PP^{3}(\KK)$ we
introduce the following notation:
$$
(Z)_{\Pi} := \sum_{C_j \subset \Pi} \alpha_j C_j \, \mbox{ and }  (Z)^{\Pi} := Z - (Z)_{\Pi}.
$$
For the proof of Thm.~\ref{thm-bound-char-p} we need the following observation.

\begin{obs} \label{obs-degree}
Let $L_1$, $\ldots$, $L_k \subset \PP^3$ be coplanar lines with $k\leq d$ and let  $\Pi = \mbox{span}(L_1, L_2)$ be the plane they span.
Moreover, assume that $\mbox{supp}(Z)$ contains none of the lines $L_1$, $\ldots$, $L_k$.
Then, the following inequality holds
$$
\deg(Z) \geq Z\cdot(L_1 + \ldots + L_k) - (k-1) \deg(Z_{\Pi})
$$
\end{obs}
\begin{proof}
Obviously, we have the equalities $\deg(Z_{\Pi}) =  Z_{\Pi}\cdot L_j$ for $j=1,\ldots,k$,
and the inequality $\deg(Z^{\Pi}) \geq Z^{\Pi}.(L_1+ \ldots + L_k)$. Therefore, we obtain
\be
   \deg(Z) & =    & \deg(Z^{\Pi})+\deg(Z_{\Pi}) \\
           & \geq & Z^{\Pi}.(L_1+\ldots+L_k) +  Z_{\Pi}.(L_1+\ldots+L_k) - (k-1) \deg(Z_{\Pi}) \, .
\ee
Since $Z=Z^{\Pi}+Z_{\Pi}$, the claim follows.
\end{proof}

Now we put:
$$
{\mathcal Z} := \flec(\XXd) - \sum_{L_j \subset \XXd} L_j
$$
Moreover, we assume $h \leq d$ to be a positive integer.

\medskip
   Our aim in this section is to prove:

\begin{mainlemma}\label{mainlemma}
\be
\deg({\mathcal Z}) \geq 6(d-3) \, .
\ee
\end{mainlemma}

   The Main Lemma then immediately implies Theorem~\ref{thm-bound-char-p} upon
   using the equality $\deg{\mathcal Z} = \deg\flec(\XXd) - \ell(\XXd)$.

\begin{definition}\mbox{}
\begin{enumerate}
\item[a)]
   We put $\ell_1(\XXd)$ (resp.  $\ell_2(\XXd)$) to denote the number of lines of that come up with multiplicity one (resp. higher than one) in $\flec(\XXd)$.
\item[b)]
   We call a line of multiplicity one in $\flec(\XXd)$  {\bf  reduced}.
\item[c)]
   We call a plane $\Pi$ {\bf $k$-spanned } if it contains $k$ reduced lines (so each $3$-spanned plane is $2$-spanned etc.).
\end{enumerate}
\end{definition}

Obviously, we have
\begin{equation} \label{eq-Zl2}
\deg({\mathcal Z}) \geq  \ell_2(\XXd) \, .
\end{equation}

As a consequence of Obs.~\ref{obs-degree} we obtain the following bound on $\deg(\mathcal Z)$.

\begin{obs} \label{obs-zzz}
Let $L_1$, $\ldots$, $L_k \subset \PP^3$ be coplanar reduced lines with $k\leq d$ and let  $\Pi = \mbox{span}(L_1, L_2)$ be the plane they span.
Then, the following inequality holds
\begin{equation} \label{eq-ineqZ}
\deg({\mathcal Z}) \geq 4k(d-3) - (k-1) \deg({\mathcal Z}_{\Pi})
\end{equation}
\end{obs}
\begin{proof} The lines $L_1$, $\ldots$, $L_k$ are  reduced, so they are no components of the support of ${\mathcal Z}$ and we can apply  Obs.~~\ref{obs-degree}  to obtain
$$
\deg({\mathcal Z}) \geq {\mathcal Z}.(L_1 + \ldots + L_k) - (k-1) \deg({\mathcal Z}_{\Pi})
$$
Recall that $L^2=-(d-2)$ for each  line $L \subset \XXd$. Thus for $q=1, \ldots, k$ we have
$$
(\flec(\XXd) - L_q).L_q = (11 d - 24) + (d - 2) =  12 d - 26 \, .
$$
By Prop.~~\ref{thm-flecnodal} the lines $L_1$, $\ldots$, $L_k$ are of the first kind.
Therefore, by Claim~\ref{claim-first-kind}, for $q = 1, \ldots, k$, we have
$$
{\mathcal Z}.L_q =  (\flec(\XXd) - L_q - \sum_{L_j \subset \XXd, j \neq q} L_j).L_q \geq (12 d - 26) - (8d - 14) = 4(d-3)  \,
$$
and the claim \eqref{eq-ineqZ} follows.
\end{proof}

\begin{obs} \label{obs-zztop}
Let $L_1$, $\ldots$, $L_k \subset \PP^3$ be coplanar reduced lines with $2 \leq k \leq d$ and let  $\Pi = \mbox{span}(L_1, L_2)$ be the plane they span. Moreover, let
$$
h \leq 4k \, .
$$
Then the following implication holds:
\begin{equation} \label{eq-ineqZZ}
\mbox{ if }  \deg({\mathcal Z}_{\Pi}) \leq \frac{(4k-h)(d-3)}{k-1} \quad \mbox{then} \quad  \deg({\mathcal Z}) \geq h(d-3)
\end{equation}
\end{obs}
\begin{proof} Insert the assumption \eqref{eq-ineqZZ} into \eqref{eq-ineqZ}.
\end{proof}

\begin{lemm} \label{lemm-threeplanes}
Let $L_1$ be a reduced line.
\begin{enumerate}
\item[a)]
If there exist three $2$-spanned planes $\Pi_1$, $\Pi_2$, $\Pi_3$ that meet along $L_1$, then
$$
\deg({\mathcal Z}) \geq 6(d-3)
$$
\item[b)]
If there exist two $3$-spanned planes $\Pi_1$, $\Pi_2$,  that meet along $L_1$, then
$$
\deg({\mathcal Z}) \geq 6(d-3)
$$
\end{enumerate}
\end{lemm}
\begin{proof} a) If   $\deg({\mathcal Z}_{\Pi_j}) \leq \frac{(4 \cdot 2 -6 )(d-3)}{2-1}$ for one of the planes  $\Pi_1$, $\Pi_2$, $\Pi_3$,
then $\deg({\mathcal Z}) \geq 6(d-3)$ by Obs.~\ref{obs-zztop} and the proof is complete. \\
Assume that
$$
\deg({\mathcal Z}_{\Pi_j}) \geq \frac{(4 \cdot 2 -6 )(d-3)}{2-1} \mbox{ for each of the planes } \Pi_1, \Pi_2, \Pi_3,
$$
then %(see Obs.~\ref{obs-function}.b)
$$
\deg({\mathcal Z}) \geq  \deg({\mathcal Z}_{\Pi_1}) + \deg({\mathcal Z}_{\Pi_2})  + \deg({\mathcal Z}_{\Pi_3}) \geq  6(d-3)
$$
and the proof is complete.

b) The claim follows as in part a).
\end{proof}
In the sequel we will also need the following bound on $\deg({\mathcal Z})$.
\begin{lemm} \label{lemm-one-reduced-line}
Let $L_1$ be a reduced line. Assume that $L_1$ is met by at most $q$ other reduced lines.
Then
$$
\deg({\mathcal Z}) \geq (6d - 13) - \frac{q}{2} \, .
$$
\end{lemm}
\begin{proof}
Obviously, we have $(\flec(\XXd)-L_1)\cdot L_1 = 12d-26$, so this gives the following bound on the number of other lines on $\XXd$ that meet $L_1$:
$$
\Big((\sum_{L_j \subset \XXd} L_j) - L_1\Big)\cdot L_1 \leq q + \frac{1}{2} (12d-26 - q) = \frac q2 + 6d -13
\,,
$$
because each non-reduced line comes with multiplicity at least two in the flecnodal divisor.
We have
$$
{\mathcal Z} = \flec(\XXd) - \sum_{L_j \subset \XXd} L_j
\,.
$$
Thus
\be
  \deg({\mathcal Z}) \geq {\mathcal Z}.L_1  &=& \flec(\XXd).L_1 - \Big(\sum_{L_j \subset \XXd} L_j - L_1\Big)\cdot L_1 - L_1^2 \\
  &\geq& (11d - 24) - (q/2 + 6d - 13) - (2-d)
\ee
and the claim follows.
\end{proof}

\begin{lemm} \label{lemm-one-line}
Let $L_1$ be a reduced line. Assume that $L_1$ is met by at most  $q_1$ other reduced lines and $q_2$ non-reduced lines. Then
$$
\deg({\mathcal Z}) \geq (12d - 26) -  (q_1 + q_2) \, .
$$
\end{lemm}
\begin{proof} Since we have
$$
\Big((\sum_{L_j \subset \XXd} L_j) - L_1\Big)\cdot L_1 = q_1 + q_2  \quad\mbox{and}\quad  {\mathcal Z} = \flec(\XXd) - \sum_{L_j \subset \XXd} L_j
\,,
$$
we get
\be
   \deg({\mathcal Z}) \geq {\mathcal Z}.L_1  &=& \flec(\XXd).L_1 -  (\sum_{L_j \subset \XXd} L_j - L_1).L_1 - L_1^2 \\
   &=& (11d - 24) - (q_1 + q_2) + (d-2)
   \,.
\ee

\end{proof}

We can now prove the Main Lemma \ref{mainlemma}.

\begin{proof}[Proof of Lemma \ref{mainlemma}.]
If $\XXd$ contains at least $6(d-3)$ lines that are not reduced, then the claim follows from  the inequality \eqref{eq-Zl2}. Therefore for the rest of the proof we assume that
the following inequality holds:
\begin{equation} \label{eq-many-non-reduced}
l_2(\XXd) < 6(d-3)
\end{equation}

\noindent
\emph{Step 1:} Assume there are no reduced lines on $\XXd$. Then % Since  $\deg(\flec(\XXd)) =  d(11d-24)$, this gives
$$
\deg({\mathcal Z}) \geq \frac12 \deg\flec(\XXd) = \frac{1}{2} d(11d-24) \geq 6(d-3) \, .
$$

\noindent
\emph{Step 2:} Suppose there exists a reduced line that is met by at most $10$ reduced lines. Then Lemma~\ref{lemm-one-reduced-line}
yields Lemma \ref{mainlemma}. Thus we can assume the following for the remainder of the proof.

\begin{state}{Assumption A}
   There exists a reduced line on $\XXd$ and each reduced line on $\XXd$ is met by at least 11 other reduced lines.
\end{state}

\noindent
\emph{Step 3:} If there exists a reduced line on $\XXd$ that is contained either in three $2$-spanned planes or in two $3$-spanned planes,
then Lemma~\ref{lemm-threeplanes}
yields the claim.  Thus we can assume that each reduced line $L_1 \subset \XXd$ forms part of one of the two configurations:
\begin{enumerate}
\item $L_1$ is contained in exactly one plane that is spanned by reduced lines and, by Assumption A, the plane in question is $k$-spanned with $k \geq 12$,
\item $L_1$ is contained in exactly two planes that are spanned by reduced lines and one of them fails to be $3$-spanned (it is only $2$-spanned).
\end{enumerate}

\noindent
\emph{Step 4:} Suppose there exists a reduced line $L_1$ that is contained in exactly one $k$-spanned plane $\Pi$, i.e.
all reduced lines that meet $L_1$  are contained in the plane $\Pi$. Let $L_2$, $\ldots$, $L_k$, with $k \geq 12$, be the reduced lines that meet
the line $L_1$. \\
Each of the lines $L_2$, $\ldots$, $L_k$ is contained in at most two $2$-spanned planes, (and if it is contained in a $2$-spanned plane $\neq \Pi$, then the other plane is
 $2$-spanned but not $3$-spanned) by Lemma~\ref{lemm-threeplanes}. Thus each of the lines $L_2$, $\ldots$, $L_k$  is met by at most one reduced line that is not contained in $\Pi$.

At most $(d-k)$ non-reduced lines are contained in $\Pi$ and we have at most $6(d-3)-1$ non-reduced lines on $\XXd$.
Our surface is smooth, so a line that is not contained in $\Pi$ meets at most one of the lines $L_1$, $\ldots$, $L_k$ (otherwise it would be in $\Pi$,
or it would meet two lines $L_{j_1}, L_{j_2}$, where $j_1 \neq j_2$,  in the same point and the latter would be a singularity of $\XXd$ because the tangent space of $\XXd$ would be too large). \\
Thus one of the lines $L_1$, $\ldots$, $L_k$
is met by at most
$$
(d-k) + \frac{1}{k}(6(d-3)-1) < (d-k) +  \frac{6}{12}(d-3)
$$
non-reduced lines.
Lemma~\ref{lemm-one-line} yields (with $q_1=k-1$ and $q_2 < (d-k) +  \frac{1}{2}(d-3)$)
the inequality
\be
   \deg({\mathcal Z})
   &\ge& (12d-26) - ((k-1)+q_2) \\
   &\geq& (12 d - 26) - (k + (d-k) +  \frac{1}{2}(d-3)) \geq  (12 d - 26) - (d+\frac{1}{2}(d-3))
\ee
which implies Lemma \ref{mainlemma}. Thus we can assume the following.

\begin{state}{Assumption B}
   Each reduced line is contained in exactly two $2$-spanned planes, one of which is
   not $3$-spanned (but only $2$-spanned).
\end{state}

\noindent
\emph{Step 5:} Let $L_1$ be a reduced line and let it be met by $k$ other reduced lines. By Assumption B the line $L_1$ is contained in two $2$-spanned planes: $\Pi_1$ and $\Pi_2$.
We can assume $\Pi_2 = \mbox{span}(L_1, L_{k+1})$ to be $2$-spanned, but not $3$-spanned. Then $\Pi_1 = \mbox{span}(L_1, \ldots, L_{k})$ is $k$-spanned but not $(k+1)$-spanned, where
$L_1$, $\ldots$, $L_{k+1}$ are assumed to be reduced lines. By Assumption A we have $k \geq 11$.

At most $(d-k)$ non-reduced lines are contained in $\Pi_1$ and we have at most $6(d-3)-1$ non-reduced lines on $\XXd$. Again, one of the lines
$L_1$, $\ldots$, $L_{k}$ is met by at most
$$
(d-k) + \frac{1}{k}(6(d-3)-1) \leq (d-k) +  \frac{6}{11}(d-3)
$$
non-reduced lines.  Moreover it is met by exactly $k$ reduced lines. As in Step 4 we obtain
$$
\deg({\mathcal Z}) \geq (12 d - 26) - (k + (d-k) +  \frac{6}{11}(d-3)) \geq  (12 d - 26) - (d+\frac{6}{11}(d-3))
$$
and the proof is complete.
\end{proof}

Finally,  we can give the proof of Theorem~\ref{thm-bound-char-p}.
\begin{proof}[Proof of Theorem~\ref{thm-bound-char-p}.]
 The claim  follows immediately  from the Main Lemma~\ref{mainlemma} and the equality $\deg{\mathcal Z} = \deg\flec(\XXd) - \ell(\XXd)$.
\end{proof}

\begin{examp} \label{example-fermat}
The Fermat surface
$$
x_0^d + x_1^d + x_2^d + x_3^d = 0
$$
contains $3d^2$ lines.
For $d \neq 4, 6, 8, 12, 20$ this is up to now the best % known
 example of a
surface
 with many lines (see e.g. \cite{sarti}).
\end{examp}

\noindent
{\it Acknowledgement.}
We would like to acknowledge the big impact of Wolf Barth
on our research.
He introduced us to the circle of ideas surrounding rational
curves on surfaces while we were postdoctoral students in
Erlangen.
This note would have never been written without his generous
support and the inspiring ideas \cite{barth-flec}
 that he communicated to
us at that time.

\end{document}